\documentclass[11pt]{amsart}

\newcommand{\popo}{\mathbb{P}^1 \times \mathbb{P}^1}
\newcommand{\pnr}{\mathbb{P}^{n_1}\times \cdots \times \mathbb{P}^{n_r}}

\newcommand{\Ix}{I_{\X}}

\newcommand{\N}{\mathbb{N}}
\newcommand{\pr}{\mathbb{P}}
\newcommand{\X}{\mathbb{X}}

\newcommand{\depth}{\operatorname{depth}}

\newcommand{\ua}{\underline{\alpha}}

\numberwithin{equation}{section}
\newtheorem{theorem}{Theorem}
\newtheorem{lemma}[theorem]{Lemma}

\theoremstyle{definition}
\newtheorem{definition}[theorem]{Definition}

\begin{document}


\title{Classifying ACM sets of points in $\popo$ via separators}
\thanks{Version: 30 May 2012}

\author{Elena Guardo}
\address{Dipartimento di Matematica e Informatica\\
Viale A. Doria, 6 - 95100 - Catania, Italy}
\email{guardo@dmi.unict.it}

\author{Adam Van Tuyl}
\address{Department of Mathematics \\
Lakehead University \\
Thunder Bay, ON P7B 5E1, Canada} \email{avantuyl@lakeheadu.ca}

\begin{abstract}
The purpose of this note is to give a new, short proof of a
classification of ACM sets of points in $\popo$ in terms of
separators.
\end{abstract}

\maketitle

Throughout this paper $k$ will denote an algebraically closed
field of characteristic zero. Given a  finite set of points $\X
\subseteq \pnr$, it can be shown (see, for example \cite[Theorem
2.1]{GVT1}) that $\dim R/I_{\X} = r$  and $1 \leq \depth R/I_{\X}
\leq r,$   where $R = k[\pnr]$ is the multigraded coordinate ring
of $\pnr$, and $I_\X$ is the multihomogeneous ideal associated to
$\X$.  When $\depth R/I_\X = r$, then we say $\X$ is {\bf
arithmetically Cohen-Macaulay} (ACM). It is natural to ask if one
can classify which finite sets of points in $\pnr$ are ACM. This
problem was initially studied in \cite{GuMaRa} in the case that
$\X \subseteq \popo$.  In particular, ACM sets of points in
$\popo$ were classified in terms of their bigraded Hilbert
functions. Other classifications exist for points in $\popo$, but
the general problem remains;  see \cite{GVT1} for details.

Marino \cite{M} gave a new classification of ACM sets of points in
$\popo$ in terms of separators of points.  Using results of
\cite{GVT,GVT1}, we will give a new, short proof of this result.
We begin by stating the necessary definitions and results.

\begin{definition}
Let $\X$ be a set of distinct points in $\pnr$ and $P \in \X$. A
multihomogeneous form $F\in R$ is a {\bf separator for} $P$ if $F(P)\neq 0$ and $F(Q)=0$ for all $Q \in \X \setminus\{P\}$.
\end{definition}

We induce a partial order on $\N^r$ by setting $(a_1,\ldots,a_r)
\succeq (b_1,\ldots,b_r)$ if $a_i \geq b_i$ for $i=1,\ldots,r$. If
$S \subseteq \N^r$ is a subset, then let $\min S$ denote the set
of minimal elements of $S$ with respect $\succeq$. \noindent
\begin{definition}
Let $\X$ be a set of distinct points in $\pnr$. The {\bf degree of a point} $P\in \X$ is
the set $\deg_{\X}(P)=\min\{ \deg F ~|~ F ~~\text{is a separator for $P\in \X$}\}.$
\end{definition}

The set $\deg_{\X}(P) = \{\ua_1,\ldots,\ua_s\} \subseteq \N^r$ may
have more than one element. If $F$ is a separator of $P$ with
$\deg F = \ua_i \in \deg_{\X}(P)$, then $F$ is essentially unique
(up to scalar multiplication).

\begin{theorem}[{\cite[Corollary 5.4]{GVT1}}]\label{lemmaA} Suppose $\deg_{\X}(P) = \{\ua_1,\ldots,\ua_s\} \subseteq \N^r$.
If $F$ and $G$ are any two separators of $P$ with $\deg F = \deg G
= \ua_i$, then there exists $0 \neq c \in k$ such that
$\overline{G} = \overline{cF} \in R/\Ix$.
\end{theorem}
As noted, one will have $|\deg_{\X}(P)| \geq 1$.  However, when $\X$ is ACM, we have:

\begin{theorem}[{\cite[Theorem 5.7]{GVT1}}]\label{theoremB} {Let $\X$ be any ACM set of points in $\pnr.$
Then for any point $P \in \X$ we have $|\deg_{\X}(P)|  = 1$.}
\end{theorem}

We will make use of the following geometric classification of ACM sets of points in $\popo$.

\begin{theorem}[{\cite[Theorem 4.3]{GVT1}}]\label{theoremC} {Let $\X \subseteq \popo$
be a finite set of points. Then $\X$ is ACM if and only if $\X$ satisfies the property:
whenever $P \times Q$ and $P' \times Q' \in \X$ with $P \neq P'$ and $Q \neq Q'$, then either $P \times Q'$ or $P' \times Q$ (or both) are in $\X$.}
\end{theorem}

The following two results found in \cite{GVT}, compute the degree of a point in some special cases.
\begin{lemma}[{\cite[Lemma 4.3]{GVT}}]\label {LemmaD} With the above notation, we have
\begin{enumerate}
\item  {Let $\{Q_1,\ldots,Q_b\}$ be $b\geq 2$ distinct points in $\pr^1$, and let $P_1$
be any point of $\pr^1$ (we allow the case the at $P_1 = Q_i$ for some $i$).
Consider the set of points \[\X = \{P_1\times Q_1,P_1 \times Q_2,\ldots,P_1 \times Q_b\} \subseteq \popo.\]
Then $\X$ is ACM, and furthermore, $\deg_{\X}(P_1\times Q_1) = \{(0,b-1)\}$.}

\item  {Let $\{P_1,\ldots,P_a\}$ be $a\geq 2$ distinct points in $\pr^1$,
and let $Q_1$ be any point of $\pr^1$ (we allow the case the at
$Q_1 = P_i$ for some $i$). Consider the set of points \[\X =
\{P_1\times Q_1,P_2 \times Q_1,\ldots,P_a \times Q_1\} \subseteq
\popo.\]Then $\X$ is ACM, and furthermore, $\deg_{\X}(P_1\times
Q_1) = \{(a-1,0)\}$.}
\end{enumerate}
\end{lemma}

\begin{theorem}[{\cite[Theorem 4.4]{GVT}}] \label{LemmaE}
Let $\X$ be an ACM set of points in $\popo$.  For any $P \times Q \in \X$ let
\[\X_{P,1} = \{P \times Q, P \times Q_2,\ldots, P\times Q_b\} \subseteq \X\]
be all the points of $\X$ whose first coordinate is $P$, and let \[\X_{Q,2} = \{P \times Q, P_2 \times Q,\ldots, P_a\times Q\} \subseteq \X\] be all the points of $\X$ whose second coordinate is $Q$.  Then \[\deg_{\X}(P\times Q) = \{(|\X_{Q,2}|-1,|\X_{P,1}|-1)\} = \{(a-1,b-1)\}.\]
\end{theorem}

We are now ready to give a new proof for Marino's main result.

\begin{theorem}[{\cite[Proposition 6.7]{M}}] Let $\X \subseteq \popo$ be a set of distinct points.
Then $\X$ is ACM if and only if $|\deg_{\X}(P)| = 1$ for all $P \in \X$.
\end{theorem}
\begin{proof} In light of Theorem \ref{theoremB}, it suffices to prove the
$(\Leftarrow)$ direction. We will prove the contrapositive
statement: if $\X$ is not ACM, then there exists a point $P \times
Q \in \X$ such that $|\deg_{\X}(P\times Q)| > 1$. Note that
throughout this proof we use that fact that if $P \times Q \in
\X$, then the defining ideal of this point is $(L_P,L_Q)$ in
$k[x_0,x_1,y_0,y_1] = k[\popo]$ where $L_P$ is a form of bidegree
$(1,0)$ and $L_Q$ is a form of bidegree $(0,1)$.

Let $\pi_1(\X) = \{P_1,\ldots,P_r\}$ and $\pi_2(\X) =
\{Q_1,\ldots,Q_s\}$ be the set of first coordinates (respectively,
second coordinates) that appear in $\X$. By Theorem
\ref{theoremC}, there exist points $P \times Q$ and $P' \times Q'$
in $\X$ such that $P \times Q'$ and $P' \times Q$ are not in $\X$.
After relabelling, we can assume $P \times Q = P_1 \times Q_1$ and
$P' \times Q' = P_2 \times Q_2$.

We set $\X_{P_1} = \{P \times Q \in \X ~|~ P = P_1\}$ and $\X_{Q_1} = \{P \times Q \in \X ~|~ Q = Q_1\}$.  We thus have
\begin{align*}
\X_{P_1} = \{P_1 \times Q_1,P_1\times Q_{i_2},\ldots,P_1 \times Q_{i_b}\} \\
\X_{Q_1} = \{P_1 \times Q_1, P_{j_2} \times Q_1,\ldots, P_{j_a}
\times Q_1\}
\end{align*}
Note that $P_1 \times Q_2 \not\in \X_{P_1}$ and $P_2 \times Q_1
\not\in \X_{Q_1}$. There are now four cases to consider. In each
case we will show $|\deg_{\X}(P_1 \times Q_1)| > 1$.

\noindent {\it Case 1:} $|\X_{P_1}| = |\X_{Q_1}| = 1$.

\noindent In this case $P_1 \times Q_1$ is the only point of $\X$
with first coordinate $P_1$ and second coordinate $Q_1$.

The two forms $F_1 = L_{P_2}L_{P_3}\cdots L_{P_r}$ and $F_1 =
L_{Q_2}L_{Q_3}\cdots L_{Q_s}$ are separators of $P_1 \times Q_1$
of degrees $(r-1,0)$ and $(0,s-1)$, respectively. (It is not hard
to see that $F_1$ and $F_2$ pass through all the points of $\X
\setminus \{P_1 \times Q_1\}$.)

If $|\deg_{\X}(P_1 \times Q_1)| = 1$, then there would be a
separator $F$ of $P_1 \times Q_1$ such that $(r-1,0) \succeq \deg
F$ and $(0,s-1) \succeq \deg F$.  But this would mean that $\deg F
= (0,0)$; however, there is no separator of $P_1 \times Q_1$ of
degree $(0,0)$.Thus $|\deg_{\X}(P_1 \times Q_1)|\!>\! 1$.

\noindent {\it Case 2:} $|\X_{P_1}| > 1$ and $|\X_{Q_1}| =1.$

\noindent The two forms $F_1 \!=\! L_{P_2}L_{P_3}\cdots
L_{P_r}L_{Q_{i_2}}L_{Q_{i_3}}\! \cdots\! L_{Q_{i_b}}$ and $F_2\!
=\! L_{Q_2}L_{Q_3}\cdots L_{Q_s}$ are two separators of $P_1
\times Q_1$ in $\X$. If $|\deg_{\X}(P_1 \times Q_1)| = 1$, then
there would exist a separator $F$ such that $(r-1,b-1) \succeq
\deg F$ and $(0,s-1) \succeq \deg F$. That is $(0,b-1) \succeq
\deg F$. Note that $F$ would also be a separator of $P_1 \times
Q_1$ in $\X_{P_1}$, and thus, by Lemma \ref{LemmaD} (1), we will
have $\deg F \succeq (0,b-1)$, In other words, $\deg F = (0,b-1)$.
Because $\deg_{\X_{P_1}}(P_1 \times Q_1) = \{(0,b-1)\}$, by
Theorem \ref{lemmaA} $F$ is the unique (up to scalar
multiplication in $R/I_{\X}$) separator of $P_1 \times Q_1$ in
$\X_{P_1}$. Now because $L_{Q_{i_2}}L_{Q_{i_3}}\cdots L_{Q_{i_b}}$
is another separator of degree $(0,b-1)$ for the point $P_1 \times
Q_1$ in $\X_{P_1}$, we have $F = cL_{Q_{i_2}}L_{Q_{i_3}}\cdots
L_{Q_{i_b}} + H$ for some nonzero scalar $c$ and $H \in I_{\X}$.
However, it then follows that $F(P_2 \times Q_2) \neq 0$,
contradicting the fact that $F$ is a separator of $P_1 \times Q_1$
in $\X$. So, $|\deg_{\X}(P)|> 1$.

\noindent {\it Case 3:} $|\X_{P_1}| = 1$ and $|\X_{Q_1}| >1.$

\noindent The proof is similar to the previous case.

\noindent {\it Case 4:} $|\X_{P_1}| > 1$ and $|\X_{Q_1}| >1.$

\noindent The two forms $$F_1 = L_{P_2}L_{P_3}\cdots
L_{P_r}L_{Q_{i_2}}L_{Q_{i_3}} \cdots L_{Q_{i_b}}$$ and $$F_2 =
L_{P_{i_2}}L_{P_{i_3}}\cdots L_{P_{i_a}}L_{Q_2}L_{Q_3}\cdots
L_{Q_s}$$ are two separators of $P_1 \times Q_1$ of degrees
$(r-1,b-1)$ and $(a-1,s-1)$, respectively. If $|\deg_{\X}(P_1
\times Q_1)| = 1$, then there would exist a separator $F$ of $P_1
\times Q_1$ with $(r-1,b-1) \succeq \deg F$ and $(a-1,s-1) \succeq
\deg F$. In other words, $(a-1,b-1) \succeq \deg F$.  Now such an
$F$ would also be a separator of $P_1 \times Q_1$ in the set of
points $\X' = \X_{P_1} \cup \X_{Q_1}$.  But then by Theorem
\ref{LemmaE}, this would mean $\deg F \succeq (a-1,b-1)$. Thus
$\deg F = (a-1,b-1)$. By Theorem \ref{LemmaE}, the point $P_1
\times Q_1$ in the scheme $\X'$ has $\deg_{\X'}(P_1 \times
Q_1)=\{(a-1,b-1)\}$, and thus by Theorem \ref{lemmaA}, the form
$F$ must be the unique (up to scalar multiplication in $R/I_{\X}$)
separator of $P_1 \times Q_1$ in $\X'$. On the other hand, the
form $F' = L_{P_{i_2}}L_{P_{i_3}}\cdots
L_{P_{i_a}}L_{Q_{i_2}}L_{Q_{i_3}} \cdots L_{Q_{i_b}}$ is also a
separator of degree $(a-1,b-1)$ of $P_1 \times Q_1$ in $\X'$, and
thus $F = cF'+H$ for some $c$ and $H\in I_\X$. But then $F(P_2
\times Q_2) \neq 0$, contradicting the fact that $F$ must pass
through every point if $\X \setminus \{P_1 \times Q_1\}$. Hence,
we must have $|\deg_{\X}(P_1 \times Q_1)| > 1$. \end{proof}

\end{document}